\documentclass{amsart}

\usepackage{amssymb, amsmath, latexsym, a4}
\usepackage[latin1]{inputenc}
\usepackage{enumerate}
\usepackage[all]{xy}

\newtheorem{De}{Definition}
\newtheorem{Th}[De]{Theorem}
\newtheorem{Pro}[De]{Proposition}
\newtheorem{Le}[De]{Lemma}

\numberwithin{equation}{subsection}

\def\ss{\sigma}
\def\xto#1{\xrightarrow[]{#1}}

\let\ss\sigma
\def\id{{\sf Id}}

\def\1{^{-1}}

\def \ker{\mathop{\sf Ker}\nolimits}

\begin{document}

\title{Crossed modules and symmetric cohomology of groups}

\author[M. Pirashvili]{Mariam  Pirashvili}

\maketitle

\begin{abstract}

This paper links the third symmetric cohomology (introduced by Staic \cite{staic_h2} and Zarelua \cite{zarelua}) to crossed modules with certain properties. The equivalent result in the language of $2$-groups states that an extension of $2$-groups corresponds to an element of $HS^3$ iff it possesses a section which preserves inverses in the $2$-categorical sense. This ties in with Staic's (and Zarelua's) result regarding $HS^2$ and abelian extensions of groups.

\end{abstract}

%
%

\section{Introduction}
Let $G$ be a group and $M$ be a $G$-module.  Symmetric cohomology $HS^*(G,M)$ was introduced by M. Staic \cite{staic} as a variant of classical group cohomology. A. Zarelua's prior definition \cite{zarelua} of exterior cohomology $H^*_\lambda(G,M)$ is very closely related to this, as shown in \cite{p}. Namely, we proved in \cite{p} that the symmetric cohomology has a functorial decomposition
$$HS^*(G,M)=H^*_\lambda(G,M)\oplus H^*_\delta(G,M),$$
where $H^i_\delta(G,M)=0$ when $0\leq i\leq 4$ or $M$ has no elements of order two.

There are natural homomorphisms $\alpha^n:HS^n(G,M)\to H^n(G,M)$ (and $\beta^n:H^n_\lambda(G,M)\to H^n(G,M)$), which are isomorphisms for $n=0,1$ and a monomorphism for $n=2$. This was shown by Staic in \cite{staic_h2}. According to our results in \cite{p}, the map $\alpha^n$ is an isomorphism for $n=2$ if $G$ has no elements of order two. In this case, $\alpha^3$ is a monomorphism. More generally, $\alpha^n$ is an isomorphism if $G$ is torsion free.
 
Classical $H^2(G,M)$ classifies extensions of $G$ by $M$. Staic showed in \cite{staic_h2} that $HS^2(G,M)$ classifies a subclass of extensions of $G$ by $M$, namely those that have a section preserving inversion. 

It is also well-known that $H^3(G,M)$ classifies the so-called crossed extensions of $G$ by $M$ \cite{L1}, or equivalently $H^3(G,M)$ classifies $2$-groups $\Gamma$, with $\pi_0(\Gamma)=G$ and $\pi_1(\Gamma)=M$. 
  
The main result of this paper is that  for groups $G$ with no elements of order $2$ the group $HS^3(G,M)$ classifies crossed extensions of $G$ by $M$ with a certain condition on the section. 
  
There has been interest in the mathematical community in exploring symmetric cohomology. Several papers have already been published on this topic, e.g. \cite{problems},\cite{singh3}, \cite{singh}, \cite{todea}.

 \section{Preliminaries}
 \subsection{Preliminaries on symmetric cohomology of groups}\label{psc}
Let $G$ be a group and $M$ be a $G$-module. Recall that the group cohomology  $H^*(G,M)$ is defined as the cohomology of the cochain complex
$C^*(G,M)$, where  the group of $n$-cochains of $G$ with coefficients in $M$ is the set of functions from $G^n$ to $M$:
$C^n(G,M)=\left\{\phi:G^n\to M\right\}$ and coboundary map  $d^n:C^n(G,M)\to C^{n+1}(G,M)$ is defined by
\begin{align*}
d^n(\phi)(g_0,g_1,\cdots ,g_n)&=g_0\cdot \phi(g_1,\cdots ,g_n)\\
&+\sum_{i=1}^n(-1)^i\phi(g_0,\cdots ,g_{i-2},g_{i-1}g_i,g_{i+1},\cdots , g_n)\\
&+(-1)^{n+1}\phi(g_0,\cdots ,g_{n-1}).
\end{align*}
A cochain $\phi$ is called \emph{normalised} if $\phi(g_1,\cdots,g_n)=0$ whenever some $g_j=1$. The collection of normalised cochains is denoted by $C^*_N(G,M)$ and the classical normalisation theorem claims that the inclusion $C^*_N(G,M) \to C^*(G,M)$ induces an isomorphism on cohomology.

  In \cite{staic} Staic introduced a subcomplex $CS^*(G,M)\subset C^*(G,M)$, whose homology is known as the symmetric cohomology of $G$ with coefficients in $M$ and is denoted by $HS^*(G,M)$. The definition is based on an action of $\Sigma_{n+1}$ on $C^n(G,M)$ (for every $n$) compatible with the differential. In order to define this action, it is enough to define how the transpositions $\tau_i = (i,i + 1)$,
$1 \leq i \leq n$ act. For $\phi \in C^n(G,M)$ one defines:
$$(\tau_i \phi)(g_1, g_2, g_3, \cdots , g_n) = \begin{cases}  -g_1\phi(g_1^{-1}, g_1g_2, g_3,\cdots , g_n), \quad {\rm if}  \ i=1,\\ 
-\phi(g_1, \cdots , g_{i-2}, g_{i-1}g_i, g_i^{-1}, g_ig_{i+1},\cdots , g_n),&  1 < i < n,\\
-\phi(g_1, g_2, g_3,\cdots , g_{n-1}g_n, g_n^{-1}) , \quad  {\rm if}  \ i=n.
\end{cases}$$
Denote by  $CS^n(G,M)$ the subgroup of   the invariants of this action. That is, $CS^n(G,M)= C^n(G,M)^{\Sigma_{n+1}}$. Staic proved  that $CS^*(G,M)$ is a subcomplex of $C^*(G,M)$ \cite{staic}, \cite{staic_h2} and hence the groups $HS^*(G,M)$ are well-defined. There is an obvious natural transformation 
$$\alpha^n:HS^n(G,M)\to H^n(G,M), \  \ n\geq 0.$$
According to \cite{staic},\cite{staic_h2},  $\alpha^n$ is an isomorphism if $n=0,1$ and is a monomorphism for $n=2$. For extensive study of the homomorphism    $\alpha^n$ for $n\geq 2$ we refer to \cite{p}.

Denote by $CS^*_N(G,M)$ the intersection $CS^*(G,M)\cap C^*_N(G,M)$. Unlike to the classical cohomology the inclusion $CS_N^*(G,M)\to CS^*(G,M)$ does not always induces an isomorphism on cohomology. The groups $H_\lambda^*(G,M)=H^*(CS_N^*(G,M)$ are isomorphic to the so called \emph{exterior cohomology of groups} introduced by Zarelua in \cite{zarelua}. According to \cite{p} the canonical map 
$$\gamma_n:H_\lambda^n(G,M)\to HS^n(G,M)$$
induced by the inclusion $CS^*_N(G,M)\to CS^*(G,M)$, is an isomorphism if $n\leq 4$, or $M$ has no elements of order two.
\subsection{Symmetric extensions and $HS^2$} It is a classical fact, that $H^2(G,M)$ classifies the extension of $G$ by $M$ and $H^3(G,M)$ classifies the crossed extensions of $G$ by $M$. One can ask what objects classify the symmetric cohomology groups $HS^2(G,M)$ and $HS^3(G,M)$. The answer to this question in the dimension two was given in \cite{staic_h2}. The aim of this work is to prove a similar result in the dimension three.

  Let $G$ be a group and $M$ be a $G$-module. Recall that an extension of $G$ by $M$ is  a short exact sequence of groups
$$0\to M\xto{i} K\xto{p} G\to 0$$
such that for any $k\in K$ and $m\in M$, one has $ki(m)k^{-1}=i(p(k)m)$. An $s$-section to this extension is a map $s:G\to K$ such that $p\circ s(x)=x$ for all $x\in G$. Let ${\bf Extgr}(G,M)$ be the category whose objects are extensions of $G$ and $M$ and morphisms are commutative diagrams
$$\xymatrix{0\ar[r] &M\ar[d]^{id}\ar[r]^i &K\ar[d]\ar[r]^{p}&G\ar[r]\ar[d]^{id}&0\\  0\ar[r]& M\ar[r]^{i'} &K'\ar[r]^{p'}&G\ar[r]&0}$$
The set of connected components of the category ${\bf Extgr}(G,M)$  is denoted by ${\sf Extgr}(G,M)$. It is well-known that there exists a natural map ${\sf Extgr}(G,M)\to H^2(G,M),$  which is a bijection. To construct this map, one needs to choose an $s$-section $s$ and then define $f\in C^2(G,M)$ by
$$s(x)s(y)=i(f(x,y))s(xy).$$
One checks that $f$ is a $2$-cocycle and its class in $H^2$ is independent of the chosen $s$-section $s$. 

Let $0\to M\xto{i} K\xto{p} G\to 0$ be an extension and $s$ be a $s$-section. Then $s$ is called \emph{symmetric} if $s(x^{-1})=s(x)^{-1}$ holds for all $x\in G$. An extension is called \emph{symmetric} if it possesses a symmetric $s$-section. The symmetric extensions form a full subcategory ${\bf ExS}(G,M)$ of the category ${\bf Extgr}(G,M)$. The set of connected components of   ${\bf ExS}(G,M)$ is denoted by  ${\sf ExS}(G,M)$. The main result of \cite{staic_h2} claims that the restriction of the bijection  ${\sf Extgr}(G,M)\to H^2(G,M)$ on ${\sf ExS}(G,M)$ yields  a bijection ${\sf ExS}(G,M)\to HS^2(G,M)$. 
 \subsection{Crossed Modules}\label{23} Recall the classical relationship between third cohomology of groups and crossed modules.
 A crossed module is a
group homomorphism  $\partial : T\to R$ together with an action of
$R$ on $T$ satisfying: 
$$\partial (^rt)=r\partial (t)r^{-1} \ {\rm and} \
 ^{\partial t}s=tst^{-1}, \ r\in R, t,s\in T.$$
 It follows from the definition that the image $Im (\partial)$ 
is a normal subgroup of $R$, and the kernel $Ker (\partial)$  is in the 
center of $T$. Moreover the action of $R$ on $T$ induces an 
action of $G$ on $Ker (\partial)$, where $G=Coker \ \partial$.

A {\it morphism } from a crossed module $\partial : T\to R$ to a crossed module $\partial' : T'\to R'$ is a 
pair of group homomorphisms $(\phi:T\to T',  \psi: R \to R')$ such 
that 
$$\psi\circ \partial = \partial '\circ \phi, \ \  
\phi (^rt)=\, ^{\psi(r)}\phi (t), r\in R, t\in T.$$
For a group $G$ and for a $G$-module $M$ one denotes by ${\bf Xext}(G,M)$ 
the category of  exact sequences 
$$0\to M\to T \xto{\partial} R\to G\to 0,$$
where $\partial:T\to R$ is  a crossed module and the 
action of $G$ on $M$ induced from the crossed module structure 
coincides with the prescribed one. The morphisms in ${\bf Xext}(G,M)$ are 
commutative diagrams
$$\xymatrix{0\ar[r] &M \ar[r]\ar[d]^{id} & T \ar[d]^\phi\ar[r]^\partial  &R \ar[r]^p\ar[d]^\psi &G \ar[r]\ar[d]_{id} & 0\\
0\ar[r] &M \ar[r] & T' \ar[r]^{\partial'}  &R' \ar[r]^{p'} &G \ar[r] & 0}$$
where $(\phi, f)$ is a morphism of 
crossed modules $(T,G,\partial) \to (T',G',\partial')$. We 
let ${\sf Xext}(G,M)$  be the class of the connected  components of the 
category ${\bf Xext}(G,M)$. Objects of the category ${\bf Xext}(G,M)$ are called \emph{crossed extensions} of $G$ by $M$.

It is a classical fact (see for example \cite{L1}) that there is a canonical bijection
$$\xi:{\sf Xext}(G,M)\to H^3(G,M).$$
The map $\xi$ has the following description. Let $0\to M \to T \xto{\partial}  R\xto{p} G\to  0$ be  a crossed extension. An \emph{$s$-section} of it is a pair of maps $(s:G\to R, \ss:G\times G\to T)$ for which the following hold
$$ps(x)=x, \quad s(x)s(y)=\partial(\ss(x,y))s(xy), \ x,y\in G.$$
An $s$-section is called \emph{normalised} if $s(1)=1$ and $\ss(1,x)=1=\ss(x,1)$ for all $x\in G$. 
It is clear that every crossed extension has a normalised s-section. Any (normalised) $s$-section $(s, \sigma)$ gives rise to a (normalised) $3$-cocycle $f\in Z^3(G,M)$ defined by
\begin{equation}\label{3cycle} 
f(x,y,z)=\,^{s(x)} \sigma(y,z)\ss(x,yz)\ss(xy,z)^{-1}\ss(x,y)^{-1}\end{equation}
To make dependence of $f$ on $\sigma$ and $s$ we sometimes write $f_\sigma$ or even $f_{s,\sigma}$ instead of $f$. Then the map $\xi$ assigns the class of $f$ in $H^3(G,M)$ to the class of
 $0\to M \to T \xto{\partial}  R\xto{p} G\to  0$ in ${\sf Xext}(G,M)$.

\section{A characterisation of symmetric cocycles}

In this section we prove the following auxilary results, which will be used in the next section. 

 \begin{Le}\label{i} i) If $\phi \in CS^n_N(G,M)$, $n\geq 2$ then $\phi(g_1, \cdots, g_n)=0$,  whenever $g_{i+1}=g_i^{-1}$ for some $1\leq i\leq n-1.$ 
\end{Le} 

\begin{proof} i)  By definition we have $(\tau_i(\phi)+\phi)(g_1,\cdots,g_n)=0$ for any $g_1,\cdots,g_n\in G$. If  $g_{i+1}=g_i^{-1}$ for some $i$, then $\tau_i(\phi)(g_1,\cdots,g_n)=0$ by the normalisation condition. Hence $\phi(g_1, \cdots, g_i,g_i^{-1},\cdots, g_n)=0$. 
 \end{proof}
The converse in general is not true, however we have the following important fact. 

\begin{Le}\label{ii}  If $\phi\in C^n_N(G,M)$ is a  cocycle, $n\geq 2$, then $\phi\in CS^n_N(G,M)$ iff  $\phi(g_1, \cdots, g_n)=0$,  whenever $g_{i+1}=g_i^{-1}$ for some $1\leq i\leq n-1.$ 

\end{Le} 

\begin{proof} Thanks to Lemma \ref{i} we need to prove that $\tau_i(\phi)+\phi=0$, if  $\phi(g_1, \cdots, g_n)=0$,  whenever $g_{k+1}=g_k^{-1}$ for some $1\leq k\leq n-1.$ 

By assumption
\begin{align*}
x_1\phi(x_2,\cdots ,x_{n+1})+\sum_{k=1}^n(-1)^k\phi(x_1,\cdots x_{k}x_{k+1},\cdots , x_{n+1})&\\
+(-1)^{n+1}\phi(x_1,\cdots ,x_{n})=0.
\end{align*}
for any $x_1,\cdots,x_{n+1}\in G$.
First we take 
$$x_k=\begin{cases} g_1, & k=1,\\ g^{-1}_1, & k=2,\\ g_1g_2,&k=3,\\ g_{k-1},& k\geq 4.\end{cases}
$$
to obtain
$$(\tau_1\phi+\phi)(g_1,\cdots, g_n)=0.$$
Next, fix $1<i<n$ and put
$$x_k=\begin{cases} g_k, & k\leq i,\\ g^{-1}_i, & k=i+1,\\ g_ig_{i+1},&k=i+2,\\ g_{k-1},&k\geq i+3.\end{cases}
$$
to obtain
$$((-1)^{i}\tau_i\phi+(-1)^{i+2}\phi)(g_1,\cdots, g_n)=0.$$
Thus $\tau_i\phi+\phi=0$, $1<i<n$. 

Finally, we take
$$x_k=\begin{cases} g_k, & k\leq n,\\ g_{n}^{-1},& k=n+1.\end{cases}
$$
to obtain
$$((-1)^{n-1}\tau_n\phi+(-1)^{n+1}\phi)(g_1,\cdots, g_n)=0.$$
Thus $\tau_n\phi+\phi=0$ and Lemma follows.
\end{proof}

In particular a cocycle  $\phi\in C^3_N(G,M)$ is symmetric (and hence defines a class in $HS^3(G,M)=H^3_\lambda(G,M)$ iff
$$\phi(x,x^{-1},y)=\phi(x,y,y^{-1})$$
for all $x,y\in G$.

Next, Lemma helps us to distinguish boundary elements in $CS^3_N(G,M)$.

\begin{Le} Suppose $\phi(x,y,z)$ is a (normalised) symmetric cocycle. Also suppose that it is a coboundary: so there exists a $g(x,y) \in C^2_N(G,M)$ such that
 $$\phi(x,y,z) = xg(y,z)-g(xy,z)+g(x,yz)-g(x,y).$$
 Then $g$ is symmetric iff $ g(x,x^{-1})=0$  for all $x\in G$.
\end{Le}
\begin{proof} Recall that $g$ is symmetric iff 
$$g(x,y)=-xg(x^{-1},xy)=-g(xy,y^{-1}).$$
If these conditions hold, we can  take $y=x^{-1}$ to obtain
$$g(x,x^{-1})=-g(1,x)=0,$$
because $g$ is normalised. Conversely, assume $g(x,x^{-1})=0$ for all $x\in G$. Since $\phi$ is symmetric we have
$$0=\phi(x,y,y^{-1})=xg(y,y^{-1})-g(xy,y^{-1})+g(x,1)-g(x,y)$$
and
$$0=\phi(x,x^{-1},z)=xg(x^{-1},z)-g(1,z)+g(x,x^{-1}z)+g(x,x^{-1}).$$
Since $g$ is normalized we have $g(x,1)=g(1,z)=0$. By assumption, we also have $g(y,y^{-1})=0=g(x,x^{-1})$. Hence,
$$g(x,y)+g(xy,y^{-1})=0 \ \ {\rm and} \ \  g(x,x^{-1}z)+g(x,x^{-1}z)=0.$$
Replacing $z$ by $xy$ in the last equality, one gets symmetric conditions on $g$.
\end{proof}

\section{Third symmetric cohomology and crossed modules}

We start with proving the following result, which links symmetric cocycles and crossed modules. Our notations are the same as at the end of Section \ref{23}.

\begin{Pro}\label{3} The class of  $0\to M \to T \xto{\partial}  R\xto{p} G\to  0$ in ${\sf Xext}(G,M)$  lies in the image of the composite map
$$HS^3(G,M)\xto{\alpha^3} H^3(G,M)\xto{\xi^{-1}} {\sf Xext}(G,M)$$
iff the crossed extension $0\to M \to T \xto{\partial}  R\xto{p} G\to  0$ has a normalised $s$-section $(s,\sigma)$ for which the following two identities hold
\begin{align*}
^{s(x)}\sigma(x^{-1},y)\sigma(x,x^{-1}y)&=\sigma(x,x^{-1}),\\
 \sigma(x,y)\sigma(xy,y^{-1})&=\,^{s(x)}\sigma(y,y^{-1}).
 \end{align*}
\end{Pro}

\begin{proof} In fact $f=f_{\sigma,\tau}$ is symmetric iff
$$f(x,x^{-1},y)=0=f(x,y,y^{-1})$$
thanks to Lemma \ref{ii}.
By definition of $f$ these conditions are equivalent to 
$$^{s(x)}\sigma(x^{-1},y)\sigma(x,x^{-1}y)\sigma(1,y)^{-1}\sigma(x,x^{-1})^{-1}=1$$
and
$$^{s(x)}\sigma(y,y^{-1})\sigma(x,1)\sigma(xy,y^{-1})^{-1}\sigma(x,y)^{-1}=1$$
Since $\sigma(1,-)=1=\sigma(-,1)$, we obtain 
$$^{s(x)}\sigma(x^{-1},y)\sigma(x,x^{-1}y)=\sigma(x,x^{-1})$$
and
$$^{s(x)}\sigma(y,y^{-1})=\sigma(x,y)\sigma(xy,y^{-1})$$
and we are done.

\end{proof}

\begin{De} A normalised $s$-section $(s,\ss)$ of a crossed extension $$0\to M \to T \xto{\partial}  R\xto{p} G\to  0$$ is called weakly symmetric if the following identities hold
\begin{itemize}
\item[i)] $s(x^{-1})=s(x)^{-1}$,
\item[ii)] $\ss(x,x^{-1})=1$, $x,y\in G$.
\end{itemize}
\end{De}

We have the following easy fact.
\begin{Le}\label{v} Let $G$ be  a group which has no elements of order two. Then  any  crossed extension $0\to M \to T \xto{\partial}  R\xto{p} G\to  0$ has a weakly symmetric $s$-section.
\end{Le}
\begin{proof} In this case $G\setminus \{1\}$ is a disjoint union of two element subsets of the form $\{x,x^{-1}\}$, $x\not=1$.  Let us choose a representative in each class. If $x$ is a representative, we set $s(x)$  to be an element in $p^{-1}(x)$. We then extend $s$ to whole $G$ by $s(1)=1$ and $s(x^{-1})=s(x)^{-1}$, where $x$ is a representative. We see that for $y=x^{-1}$, one has  $s(x)s(y)s(xy)^{-1}=1$. Thus one can choose $\ss$  with property $\ss(x,x^{-1})=1$ and lemma follows. 
\end{proof}
\begin{De}\label{6} A weakly symmetric  $s$-section $(s,\ss)$ of a crossed extension $$0\to M \to T \xto{\partial}  R\xto{p} G\to  0$$  is called  symmetric if the following identities hold
\begin{itemize}
\item[i)] $\ss(x,y) \cdot \, ^{s(x)}\ss(x^{-1},xy)=1$, $x,y\in G$.
\item[ii)] $\ss(x,y)\cdot \ss(xy,y^{-1})=1$, $x,y\in G$.
\end{itemize}
A crossed extension $0\to M \to T \xto{\partial}  R\xto{p} G\to  0$ is called symmetric if it has a symmetric $s$-section.
\end{De}


Symmetric crossed extensions of $G$ by $M$ form a subset ${\sf XextS}(G,M)$ of the set of ${\sf Xext}(G,M)$.

\begin{Th} Let $G$ be a group which has no elements of order two, then there is a natural bijection
$$HS^3(G,M)\cong {\sf XextS}(G,M).$$
\end{Th}


\begin{proof} By Lemma \ref{v} any crossed extension has a weakly symmetric $s$-section $(s,\ss)$. By Proposition \ref{3} the corresponding $3$-cocycle is 
symmetric if 
\begin{align*}
^{s(x)}\sigma(x^{-1},y)\sigma(x,x^{-1}y)&=1,\\
 \sigma(x,y)\sigma(xy,y^{-1})&=1.
 \end{align*}
 Now, if we replace $x$ by $x^{-1}$ in the first identity and then act by $s(x)$, we see that these conditions are exactly ones in the Definirion \ref{6}. Hence by Proposition \ref{3}, the image of the composite map $$HS^3(G,M)\xto{\alpha^3} H^3(G,M)\cong   {\sf Xext}(G,M)$$ is exactly ${\sf XextS}(G,M)$. On the 
 other hand,  since $G$ has no elements of order two, the map $\alpha^3$ is injective. This follows from the part ii) of Corollary 4.4 \cite{p}, because $HS^3=H^4_\lambda$, thanks to Theorem 3.9 \cite{p}. It follows that the induced map $HS^3(G,M)\to    {\sf Xext}(G,M)$ is a bijection.
\end{proof}

\section{Interpretation in terms of 2-groups} For us (strict) 2-groups are group objects in the category of small categories. Thus it is a category $\sf C$ (in fact a groupoid) equipped with a  bifunctor $\cdot:\sf C \times C\to C$, $(a,b)\mapsto a\cdot b$ which is strictly associative and satisfies group object axioms. These  objects are also known under the name categorical groups and 1-cat-groups see, \cite{L2}. Recall the relationship between crossed modules and 2-groups \cite{L2}. Let $\partial:T\to R$ be a crossed module. It defines a 2-group $\mathsf{Ca}_{T\to R}$. Objects of  $\mathsf{Ca}_{T\to R}$ are  elements of $R$. A morphism from $r\in R$ to $r'\in R$ is an element $t\in T$ such that
$$r'=\partial(t)r,$$
In this situation we use the notation $r\xto{t}r'$. The composite of arrows $r\xto{t}r'\xto{t'}r''$ is $r\xto{t't}r''$. It is clear that $r\xto{1} r$ is the identity arrow ${\id}_r$ in the category $\mathsf{Ca}_{T\to R}$.  Any morphism in  $\mathsf{Ca}_{T\to R}$ is an isomorphism. The inverse of $r\xto{t}r'$ is $r'\xto{t^{-1}}r$.  As usual we set $M=\ker (\partial)$. Observe that  any $m\in M$ defines  an endomorphism $r\xto{m}r$ of $r\in R$ and conversely, any endomorphism  of $r$ has this form.

The bifunctor $$ \mathsf{Ca}_{T\to R} \times \mathsf{Ca}_{T\to R} \xto{\cdot} \mathsf{Ca}_{T\to R}$$
given on objects by the multiplication rule in the group $R$, while on morphisms it is given by
$$(r\xto{t} z) \cdot (x\xto{s} y)=rx\xto{t(\,^zs)} zy, \ \ r,x,y,z\in R, s,t\in T.$$
In particular, we have
$$(x\xto{t} y) \cdot \id_z =xz\xto{t} yz, $$ 
$${\id}_r \cdot (x\xto{t}y)=rx\xto{^rt} ry.$$

It is well-known that any 2-group is isomorphic to the 2-group of the form $\mathsf{Ca}_{T\to R}$. For a uniquely defined (up to isomorphism) crossed module $\partial:T\to R$, see for example \cite{L2}.

In particular crossed modules gives rise to monoidal categories. So, we can consider monoidal functors. We recall the corresponding definition. Let 
$\sf C$ and $\sf D$ be 2-groups. An \emph{$s$-functor} $\sf C\to D$ is a pair $(F,\xi)$, where $F:\sf C\to D$ is  a functor and 
$\xi$ is a natural transformation from the composite functor  ${\sf C \times C} \xto{\cdot} {\sf C} \xto{F} \sf D$ to the composite funtor ${\sf C \times C} \xto{F\times F} {\sf D\times D} \xto{\cdot} \sf D$
Thus for any objects $x$ and $y$ of $\sf C$ we have a morphism $\xi_{x,y}:F(x\cdot y) \to F(x)\cdot F(y)$, which is natural in $x$ and $y$. 

In what follows, we will assume that $(F,\xi)$ is normalised, meaning that $F(1)=1$ and $\xi_{x,y}=\id$, if $x=1$ or $y=1$. Thus for any object $x$ we have a morphism $\xi(x,x^{-1}):1\to F(x)\cdot F(x^{-1})$, which will be play an important role later. 

  An $s$-functor is \emph{monoidal} if for any objects $x,y,z$ of the category $\sf C$ the  diagram
$$\xymatrix{F(x\cdot y\cdot z)\ar[rr]^{\xi_{x\cdot y,z}}\ar[d]^{\xi_{x,y\cdot z}}
&&F(x\cdot y)\cdot F(z)\ar[d]^{\xi_{x,y}\cdot \id_{F(z)}} \\ F(x)\cdot F(y\cdot z)\ar[rr]_{\id_{F(x)}\cdot \xi_{y, z}}
 &&F(x)\cdot F(y)\cdot F(z)}$$
commutes.
 
 Let  $0\to M \to T \xto{\partial}  R\xto{p} G\to  0$ be a crossed extension.  In this situation we have two 2-groups ${\sf Ca}_{T\to R}$ and ${\sf Ca}_G$. The second one is $G$ considered as a discrete category (equivalently, the $2$-group, corresponding to the crossed module $1\to G$). The homomorphism $p$ yields the strict monoidal functor ${\sf Ca}_{T\to R} \to {\sf Ca}_G$, which is still denoted by $p.$  
 
 One can consider sections of $p$. We will ask different level of compatibility of sections with monoidal structures. The weakest condition to ask to such a section is to be a functor. Since ${\sf Ca}$ is a discrete category, we see that such a section of the functor $p$ is nothing but a set section of the map $p:R\to G$. Next, is to ask to the functor ${\sf Ca}_G\to {\sf Ca}_{T\to R}$ to be an $s$-functor. Call them $s$-sections of $p$. One easily, observes that there is a one-to-one correspondence between $s$-sections $(F,\xi)$ of the functor $p$ and $s$-sections of a crossed extension $0\to M \to T \xto{\partial}  R\xto{p} G\to  0$. In fact, if $(s,\sigma)$ is an $s$-section,
then $(F,\xi)$ is an $s$-functor  ${\sf Ca}_G \to {\sf Ca}_{T\to R}$, for which  $F\circ p=\id_{{\sf Ca}_G}$. Here the functor $F$ and natural transformation $\xi$ are defined as follows. Since ${\sf Ca}_G$ is  a discrete category, the functor $F$ is uniquely determined by the rule:
$$F(g)=s(g), \ g\in G.$$ 
Next, the natural transformation $\xi$ is uniquely determined by the  family of morphisms 
$$\xi_{g,h}=\left(s(gh)\xto{\sigma(g,h)} s(g)s(h)\right), \ g,h\in G.$$
Even stronger assumption is to ask to the pair $(F,\xi)$ to be a monoidal functor. As the following  well-known result shows this condition is a 2-dimensional analogue of a splitting in a short exact sequence.
\begin{Pro} The class of a crossed extension 
$$0\to M \to T \xto{\partial}  R\xto{p} G\to  0$$
is zero iff the strict monoidal functor $p:{\sf Ca}_{T\to R} \to {\sf Ca}_G$ has a section ${\sf Ca}_G \to {\sf Ca}_{T\to R}$, which is monoidal. That is there exists  an $s$-section $(s,\sigma)$ of  $0\to M \to T \xto{\partial}  R\xto{p} G\to  0$, for which the corresponding $s$-functor $(F,\xi)$ is monoidal.
\end{Pro}

Since  we did not find appropriate reference we give the proof.

\begin{proof} Let $(s,\sigma)$ be  an $s$-section of $0\to M \to T \xto{\partial}  R\xto{p} G\to  0$. Then the diagram in the definition of the monoidal functor for $(F,\xi)$ has the form
$$\xymatrix{s(x\cdot y\cdot z)\ar[rr]^{\sigma(x\cdot y,z)}\ar[d]^{\sigma(x,y\cdot z)}
&&s(x\cdot y)\cdot s(z)\ar[d]^{\sigma(x,y)} \\ s(x)\cdot s(y\cdot z)\ar[rr]_{^{s(x)} \sigma(y, z)}
&&s(x)\cdot s(y)\cdot s(z)}$$
Thus commutativity of this diagram is equivalent to the vanishing of the 3-cocycle  $f$ in (\ref{3cycle}). 
So, if part is done. Conversely, if $f_\sigma$ defined in (\ref{3cycle}) is a coboundary and
$$f_\sigma(x,y,z)=xk(y,z)-k(xy,z)+k(x,yz)-k(x,y)$$
for a function $k:G^2\to M$, then $(s,\tau)$ is also an $s$-section for which $f_\tau=0$. Here $\tau(g,h)=\sigma(g,h)-k(g,h)$.
\end{proof}

We now introcuce another conditions to an $s$-functor $(F,\xi)$, which are weaker then the monoidal functor. 

\begin{De} An $s$-functor $(F,\xi)$ is called symmetric, if 
$$\xymatrix{
F(x)\ar[rrr]^{\xi(xy,y^{-1})}\ar[drrr]_{\id_{F(x)}\cdot \xi(y,y^{-1})}
&&&F(xy)\cdot F(y^{-1})\ar[d]^{\xi(x,y)\cdot \id_{F(y^{-1})}}\\
&&&F(x)\cdot F(y)\cdot F(y^{-1})
}
$$
and
$$\xymatrix{
F(y)\ar[rrr]^{\xi(x,x^{-1}y)} 
\ar[drrr]_{\xi(x,x^{-1})\cdot \id_{F(y)}}
&&&F(x)\cdot F(x^{-1}y)\ar[d]^{\id_{F(x)}\cdot \xi(x^{1},y)}\\
&&&F(x)\cdot F(x^{-1})\cdot F(y)
}
$$

\end{De}

Then we have the following obvious facts: 

\begin{itemize}
\item Any monoidal functor is symmetric $s$-functor. 

\item The $s$-functor corresponding to an $s$-section $(s,\ss)$ is symmetric, if $(s,\ss)$ satisfies the condition listed in Proposition \ref{3}.
\end{itemize}

Thus  2-groups, for which the corresponding class in $H^3(G,M)$ lies in the image of  $HS^3(G,M)$ can be characterized, as those, for which there exists a symmetric $s$-functor  ${\sf Ca}_G \to {\sf Ca}_{T\to R}$, which is a section of $p:{\sf Ca}_{T\to R}\to {\sf Ca}_G $.

\end{document}